\documentclass[pdflatex,sn-mathphys-num]{sn-jnl}

\usepackage{graphicx}%
\usepackage{multirow}%
\usepackage{amsmath,amssymb,amsfonts}%
\usepackage{amsthm}%
\usepackage{mathrsfs}%
\usepackage[title]{appendix}%
\usepackage{xcolor}%
\usepackage{textcomp}%
\usepackage{manyfoot}%
\usepackage{booktabs}%
\usepackage{algorithm}%
\usepackage{algorithmicx}%
\usepackage{algpseudocode}%
\usepackage{listings}%

\theoremstyle{thmstyleone}%
\newtheorem{theorem}{Theorem}
\newtheorem{lemma}[theorem]{Lemma}%
\newtheorem{proposition}[theorem]{Proposition}%

\theoremstyle{thmstyletwo}%
\newtheorem{remark}{Remark}%

\theoremstyle{thmstylethree}%
\newtheorem{definition}{Definition}%

\def\O{\Omega}

\def\n{\nabla}

\def\n{\nabla}

\def\O{\Omega}

\def\n{\nabla}
\def\<{\langle}
\def\>{\rangle}

\def\n{\nabla}

\def\O{\Omega}

\def\rr{\mathbb{R}}
\def\ve{\varepsilon}

\begin{document}

\title[General monotone formula for homogeneous $k$-Hessian equation in the exterior domain and its applications]{General monotone formula for homogeneous $k$-Hessian equation in the exterior domain and its applications}


\author[1]{\fnm{Jiabin} \sur{Yin}}\email{jiabinyin@126.com}
\equalcont{These authors contributed equally to this work.}

\author*[2,3]{\fnm{Xingjian} \sur{Zhou}}\email{zhouxingjian@stu.xmu.edu.cn}
\equalcont{These authors contributed equally to this work.}

\affil*[1]{\orgdiv{School of Mathematics and Statistics}, \orgname{Xinyang Normal University}, \orgaddress{\city{Xinyang}, \postcode{464000}, \state{Henan}, \country{P.R. China}}}

\affil[2]{\orgdiv{School of Mathematical Sciences}, \orgname{Xiamen University}, \orgaddress{\city{Xiamen}, \postcode{361005}, \state{Fujian}, \country{P.R. China}}}

\affil[3]{\orgdiv{Department of Mathematics}, \orgname{University of Trento}, \orgaddress{\street{via Sommarive 14}, \city{Povo}, \postcode{38123}, \state{Trento}, \country{Italy}}}


\abstract{In this paper, we deal with an overdetermined problem for the $k$-Hessian equation ($1\leq k<\frac n2$) in the exterior domain and  prove the corresponding ball characterizations. Since that Weinberger type approach seems to fail to solve the problem, we give a new perspective to solve exterior overdetermined problem by combining two integral identities and geometric inequalities inspired by Brandolini-Nitsch-Salani's results \cite{BNS}. Meanwhile, we establish general monotone formulas to derive geometric inequalities related to $k$-admissible solution $u$ in $\mathbb R^n\setminus\Omega$, where $\Omega$ is smooth, $k$-convex and star-shaped domain, which  constructed by Ma-Zhang\cite{MZ} and Xiao\cite{xiao}.}




\maketitle

\section{Introduction}\label{sec1}

Over recent decades, level set methods have emerged as a powerful tool for analyzing geometric phenomena. A now classical application lies in establishing geometric inequalities through PDE techniques. In their seminal work, Agostiniani and Mazzieri \cite{AM} investigated the exterior Dirichlet problem for the Laplace equation, thereby proving fundamental inequalities including the Willmore inequality. Subsequently, Fogagnolo, Mazzieri, and Pinamonti \cite{FMP} extended this approach to the $p$-Laplacian equation, deriving the volumetric Minkowski inequality via analysis of its exterior Dirichlet problem. A significant advancement was achieved by Agostiniani, Fogagnolo, and Mazzieri \cite{AFM}, who eliminated the convexity requirement for domains through innovative use of monotonicity formulas along solution level sets.
It is worth highlighting that their approach can be used  to obtain sharp geometric inequalities for hypersurfaces in manifolds with nonnegative Ricci curvature \cite{AFM2020,BFM2024,FM} as well as for static manifolds \cite{AM2017,AMO2022,Miao2024}, for  asymptotically flat manifolds \cite{AMMO,AMO2024,XYZ} and so on.

Recently, Ma-Zhang \cite{MZ}  study the exterior Dirichlet problem for the homogeneous $k$-Hessian equation for $1\leq k< \frac n2$:
\begin{equation}\label{eqn:1.1}
	\left\{
	\begin{aligned}
		S_k(\nabla^2 u)=&0\ \ {\rm in}\ \ \mathbb R^n\setminus\bar{\Omega},\\
		u=&-1\ \ {\rm on} \ \ \partial\Omega,\\
		u(x)\rightarrow& 0\ \ {\rm as}\ \ |x|\rightarrow\infty.
	\end{aligned}\right.
\end{equation}
If $\Omega$ is convex domain and strictly $(k-1)$-convex, they proved that the existence, regularity of  the solution of \eqref{eqn:1.1}. In particular, they get an weighted geometric inequality, for $a\geq\frac{k(n-k-1)}{n-k}$,
$$
\frac{n-k}{n-2k}\int_{\partial\O}|\n u|^{a+1}H_{k-1}\leq\int_{\partial\O}|\n u|^aH_k,
$$ which is a natural generalization of the $k = 1$ case by using almost monotonicity formula. In the meanwhile,  Xiao \cite{xiao} solved the exterior Dirichlet problem for the homogenous $k$-Hessian equation under the assumption condition  $(k-1)$-convex and star-shaped domain. Then, Xiao obtained generalized Minkowski inequality
$$
\int_{\partial\Omega}|\nabla u|^{n-k}H_{k-1}\geq\binom{n-1}{k-1}(\frac{n}{k}-2)^{n-k}
|\mathbb S^{n-1}|.
$$ Ma-Zhang and Xiao both consider the following quantity
\begin{equation}\label{eqn:1.5}
	\int_{\{u=t\}}H_{k-1}|\nabla u|^{b+1}(-u)^{(\frac{n-k}{2k-n})a}.
\end{equation}
Then, they  prove that \eqref{eqn:1.5} is monotonic.

Inspired by  Ma-Zhang and Xiao' idea, we can consider the following quantity, for $a\geq\frac{k(n-k-1)}{n-k}$,
\begin{equation}\label{eqn:1.6}
	F(t)=C_1(t)\int_{\{u=t\}}H_{k}|\nabla u|^{a}+C_2(t)\int_{\{u=t\}}H_{k-1}|\nabla u|^{a+1}.
\end{equation}
Thus, we obtain the  first results of this paper as follows.
\begin{theorem}\label{thm:1.2}
	Assume $1\leq k<\frac n2$. Let $\Omega\subset\mathbb R^n$ be a $k$-convex, star-shaped domain with smooth boundary $\Sigma$ and $u\in C^{1,1}(\mathbb R^n\setminus\Omega)$ be a unique solution of \eqref{eqn:1.1}. Let
	\begin{equation*}
		\begin{aligned}
			C_1(t)=&(-t)^{-\frac{(a-k)(n-k)+k}{n-2k}}C_3+(-t)^{1-\frac{(a-k)(n-k)+k}{n-2k}}C_4\\  C_2(t)=&-\frac{(a-k)(n-k)+k}{(n-2k)(a+1-k)}C_3(-t)^{-\frac{(a-k+1)(n-k)}{n-2k}}-\frac{n-k}{n-2k}C_4(-t)^{1-\frac{(a-k+1)(n-k)}{n-2k}},
		\end{aligned}
	\end{equation*}
	where  $C_1(t)\geq0$ and $C_4,C_3\in\mathbb R$.  Then, for any $a\geq\frac{k(n-k-1)}{n-k}$
	\begin{equation}
		F(t)= C_1(t)\int_{\Sigma_t} H_k |\nabla u|^a+C_2(t)\int_{\Sigma_t} H_{k-1} |\nabla u|^{a+1}
	\end{equation}
	is non-increasing for $t\in[-1,0)$. Moreover,
	\begin{equation}\label{eqn:1.8}
		\begin{aligned}
			F(t)\geq\frac{n-2k}{k(a+1-k)}\binom{n-1}{k-1}(\frac{n}{k}-2)^{a}
			\rho^{ k(n-k-a-1)}|\mathbb S^{n-1}|C_3
		\end{aligned}
	\end{equation}
	for any $t\in[-1,0)$, the equality holds if and only if
	$\Sigma_t$ is a sphere.
\end{theorem}
As a consequence, we obtain the following geometric inequalities.
\begin{theorem}\label{thm:1.3}
	Assume $1\leq k<\frac n2$. Let $\Omega\subset\mathbb R^n$ be a $k$-convex, star-shaped domain with smooth boundary $\Sigma$ and $u\in C^{1,1}(\mathbb R^n\setminus\Omega)$ be a unique solution of \eqref{eqn:1.1}.  Then the following inequalities hold, for $a\geq\frac{k(n-k-1)}{n-k}$:
	\begin{equation}\label{eqn:1.9}
		\frac{n-k}{n-2k}\int_{\Sigma}|\n u|^{a+1}H_{k-1}\leq\int_{\Sigma}|\n u|^aH_k.
	\end{equation}
	\begin{equation}\label{eqn:1.10}
		\int_{\Sigma}|\nabla u|^{n-k}H_{k-1}\geq\binom{n-1}{k-1}(\frac{n}{k}-2)^{n-k}
			|\mathbb S^{n-1}|.
	\end{equation}
	Moreover, equality in each of the above inequalities holds  if and only if
	$\Omega$ is a ball.
\end{theorem}
\begin{remark}
	Let $a=n-k-1$, from \eqref{eqn:1.9} and \eqref{eqn:1.10}, we have
	$$
	\int_{\Sigma}|\n u|^{n-k-1}H_k\geq\binom{n-1}{k-1}(\frac{n}{k}-2)^{n-k-1}\frac{n-k}{k}|\mathbb S^{n-1}|.
	$$
\end{remark}

Finally, we consider the study of symmetry in overdetermined boundary value problems. The main technique
to tackle such problems are the celebrated method of moving planes developed by Alexandrov \cite{Al,Al1} and Serrin \cite{Se} as well as  Weinberger's approach \cite{Wei} which is based on maximum principle for so-called $P$-function and Pohozaev's integral identity.

There are plenty of considerations for different kinds of overdetermined  boundary value problems. For our purpose, we recall a result of Reichel \cite{R}, who considered an overdetermined problem for capacity in an exterior domain. The capacity of a smooth bounded domain $\O\subset \rr^n (n\ge 2)$ is defined as
$${\rm Cap}(\O)={\rm inf}\Big\{
\int_{\mathbb R^n}|\nabla v|^2dx\,\Big|\, v\in C_{c}^{\infty}(\mathbb R^n),\, v\geq1\ \ {\rm on}\ \ \Omega\Big\}.
$$
The minimizer for ${\rm Cap}(\O)$ is characterized by the capacitary potential $u$ satisfying
\begin{equation}\label{eqn:1.1'}
	\left\{
	\begin{aligned}
		\Delta u=&0\ \ {\rm in}\ \ \mathbb R^n\setminus\bar{\Omega}\\
		u=&1\ \ {\rm on} \ \ \partial\Omega\\
		u\rightarrow& 0\ \ {\rm as}\ \ |x|\rightarrow\infty,
	\end{aligned}\right.
\end{equation}
Reichel \cite{R}  considered the problem \eqref{eqn:1.1'} with an extra boundary condition
\begin{equation}\label{eqn:1.2'}
	|\nabla u|=c  \ \ {\rm on} \ \ \partial\Omega,
\end{equation}
and proved that \eqref{eqn:1.1'} and \eqref{eqn:1.2'} admits a solution if and only if $\O$ is a ball.
Reichel's proof is again based on the method of moving planes and he also extended  in \cite{R1} such result to more general equations involving $p$-capacity in an exterior domain.
Garofalo-Sartori \cite{GS} and Poggesi \cite{P} reproved Reichel's result for $p$-capacity by using Weinberger type approach.

We will study the problem \eqref{eqn:1.1} with the overdetermined condition
\begin{equation}\label{eqn:1.2}
	|\nabla u|=c\ \ {\rm on} \ \ \partial\Omega,
\end{equation}
for some constant $c>0$. However, Weinberger type approach seems to fail to solve the  problem \eqref{eqn:1.1} and \eqref{eqn:1.2} for $k\geq2$. Inspired by Brandolini-Nitsch-Salani's results \cite{BNS}, we adopt new method to solve it and  obtain the following results  when the domain is assumed to be convex, by using geometric inequalities.
\begin{theorem}\label{thm:1.1}
	Let $1\le k<\frac n2$ and $\Omega\subset\mathbb R^n$ be a smooth bounded convex domain. Then \eqref{eqn:1.1} and \eqref{eqn:1.2} admits a  solution if and only if $\Omega$ is a  ball.
\end{theorem}

\begin{remark}~
\begin{itemize}
\item [(1)]	We  adopt some integral identities and inequality \eqref{eqn:1.9} ( Ma-Zhang's inequality of \cite{MZ}) to prove theorem \ref{thm:1.1}. Also, we used some special Aleksandrov-Fenchel inequalities which will hold under the assumption $\Omega$ is smooth bounded convex. 
\item [(2)] If one could get these special Aleksandrov-Fenchel inequalities under $k-$convex  and star-shaped assumption, then the convex condition of theorem \ref{thm:1.1} can be replaced by $k-$convex and star-shaped.
\item [(3)] In the Euclidean setting, we generalize Brandolini-Nitsch-Salani's $k=1$ results \cite{BNS} to the range $1\leq k<n/2$.
\end{itemize}
\end{remark}

\section{Preliminaries}\label{sec2}

\subsection{Symmetric function}
In this section, we review fundamental concepts and properties for the $k$-Hessian
operators. For the proof of the facts below, we refer to \cite{BNST,Garding,Guan-Li,Reilly}.

Throughout the paper we adopt the Einstein summation convention for repeated indices.

For $k\in \{1,\ldots, n\}$, the $k$-th elementary
symmetric function of $\lambda=(\lambda_1,\ldots,\lambda_n\in\mathbb R^n$ is defined by
$$
S_k(\lambda):=\sum_{1\leq i_{1}<\cdots<i_k\leq n}\lambda_{i_1}\cdots\lambda_{i_k}.
$$
Given  a real symmetric $n\times n$ matrix $A=(a_{ij})$ with eigenvalues $\lambda(A)$, we can define the $k$-th elementary symmetric polynomial by $S_k(A):=S_k(\lambda(A))$. Thus
$$
S_k(A)=\frac1{k!}\delta_{j_1,\cdots, j_k}^{i_1,\cdots, i_k}a_{{i_1}{j_1}}\cdots a_{i_kj_k},
$$
where $\delta_{j_1,\cdots, j_k}^{i_1,\cdots,i_k}$ is the generalized Kronecker symbol defined by
\begin{equation*}
	\delta_{j_1,\cdots, j_k}^{i_1,\cdots,i_k}:=
	\left\{
	\begin{aligned}
		&1,\ \  \ \ \ {\rm if}\ (i_1,\cdots,i_k)\ \ {\rm is \ an \ even \ permutation\ of} \  (i_1,\cdots,i_k), \\
		&-1,\ \  \ \ \ {\rm if}\ (j_1,\cdots,j_k)\ \ {\rm is \ an \ odd \ permutation\ of} \  (j_1,\cdots,j_k), \\
		&0,\ \  \ \ \ {\rm otherwise}.
	\end{aligned} \right.
\end{equation*}
We use the convention that $S_0= 1$ and $S_k=0$ for $k>n$.
Denote
$$
S_{k}^{ij}(A):=\frac{\partial S_k(A)}{\partial a_{ij}}=\frac1{(k-1)!}\delta_{j_1,\cdots, j_{k-1},j}^{i_1,\cdots,i_{k-1},i}a_{i_1j_1}\cdots a_{i_{k-1}j_{k-1}}.
$$

The following properties of $S_k(A)$ can be found in \cite{Reilly}, see also \cite{DGX1} for non-symmetric matrices.
\begin{proposition}\label{prop:2.01}
	For any $n\times n$ matrix $A$, we have
	\begin{equation}\label{eqn:2.01}
		\left\{
		\begin{aligned}
			S_k^{ij}(A)=&S_{k-1}(A)\delta_{ij}-\sum_{l=1}^nS_{k-1}^{il}(A)a_{jl},\\
			S_k^{ij}(A)h_{i}^{k}h_{k}^{j}=&S_1(A)S_k(A)-(k+1)S_{k+1}(A),\\
			S_{k}^{ij}(A)\delta_{i}^{j}=&(n-k+1)S_{k-1}(A).
		\end{aligned}\right.
	\end{equation}
\end{proposition}

Now we collect some inequalities related to the elementary symmetric functions.
The Garding cone $\Gamma^+_{k}$ is defined as
$$
\Gamma^+_{k}=\{\lambda\in\mathbb R^n|\, S_i>0,\ {\rm for}\ 1\leq i\leq k\}.
$$
We say $A$ belongs to $\Gamma^+_k$ if its eigenvalue $\lambda(A)\in \Gamma^+_k$.
For $A=(a_{ij})\in\Gamma_k$ and $1\leq m\leq l\leq k$,
\begin{equation}
	\left(\frac{S_m(A)}{\binom{n}{m}}\right)^{\frac1m}\geq\left(\frac{S_l(A)}{\binom{n}{l}}\right)^{\frac1l},
\end{equation}
and  the equality holds if and only if $\lambda_1=\lambda_2=\cdots=\lambda_n$.
\begin{proposition}\label{prop:2.1}
	Let $u$ be a $C^{2}$ function on $\mathbb{R}^{n}$, then
	
	\begin{itemize}
		\item [(i)] if $\lambda(\nabla^{2}u)\in \Gamma^+_k$, then $\big(\frac{\partial }{\partial u_{ij}}S_k(\nabla^{2}u)\big)$ is positive definite and $(S_k(\nabla^{2}u))^{\frac1k}$ is concave with respect to $\nabla^{2}u$.
		\item [(ii)] $S_k^{ij}(\nabla^{2}u)$ is divergence free and $S^{ij}_k(\nabla^{2}u)u_{il}=S_k^{il}(\nabla^{2}u)u_{ij}$.
	\end{itemize}
\end{proposition}

\vskip 2mm
\subsection{Mean curvatures and Minkowskian integral formula}~

Let $\Omega$ be a bounded connected domain of $\mathbb R^n$ of class $C^2$ having principal
curvatures $\kappa= (\kappa_1,\ldots,\kappa_{n-1})$ and outer unit normal $\nu$. For $k = 1,\ldots, n-1$ we
define the $k$-th curvature of $\partial\Omega$ by
$$
H_k(\partial\Omega)=S_k(\kappa_1,\ldots,\kappa_{n-1}),
$$
moreover we set
$$H_0 = S_0\equiv 1, H_n \equiv 0.$$
For example, $H_1$ is equal to $n-1$ times the mean curvature of $\partial\Omega$, while $H_{n-1}$ is
the Gauss curvature of $\partial\Omega$. In analogy with the case of functions, $\Omega$ is said $k$-convex, with $k\in \{1,\ldots, n-1\}$, if $H_j\geq0$ for $j = 1,\ldots, k$ at every point $x\in\partial\Omega$. We recall here that any sublevel set of a $k$-convex function is $(k-1)$-convex (see \cite{CNS}).

Then, we recall the following identities, known in differential geometry and
in the theory of convex bodies as Minkowskian integral formulae (see \cite{H,S}, for
instance):
\begin{equation}\label{eqn:2.3}
	\int_{\partial\Omega}\langle x,\ \nu\rangle H_{k}=\frac{n-k}{k}\int_{\partial\Omega}H_{k-1}.
\end{equation}

\vskip 2mm
\subsection{Curvatures of level sets and Hessian operators}~

Let $u$ be a $C^2$ function, let $t$ be a regular value of $u$ and let $L = \{u\leq t\}$. If $H_k$
stands for the $k$-th curvature of the level set $\partial L$ at the point $x$, then it is well-known
that
$$
S_1(\nabla^{2}u)=H_1|\nabla u|+\frac{u_iu_ju_{ij}}{|\nabla u|^2},
$$
that means that the value of $\Delta u$ at a point only involves derivatives of $u$ along
the line of steepest descent passing through that point and the mean curvature
$H_1/(n-1)$ of the level surface through the point. More generally, for $1\leq k\leq n$,
a straightforward calculation yields
\begin{equation}\label{eqn:2.4}
	S_k(\nabla^{2}u)=H_k|\nabla u|^k+\frac{S_k^{ij}(\nabla^{2}u)u_iu_lu_{lj}}{|\nabla u|^2},
\end{equation}
Finally, for $1\leq k\leq n$, the following pointwise identity holds (see \cite{Reilly,T})
\begin{equation}\label{eqn:2.5}
	H_{k-1}=\frac{S_{k}^{ij}(\nabla^{2}u)u_iu_j}{|\nabla u|^{k+1}}.
\end{equation}
%


\section{The proof of Theorem}\label{sec3}

In this section, we will prove the Theorem \ref{thm:1.2} and \ref{thm:1.3}.

Firstly, We will collect some facts for $k$-Hessian equation \eqref{eqn:1.1}.
\begin{definition}
	For any open set $U\subset\mathbb R^n$ a function $u \in C^{1,1}$ is called $k$-admissible if the eigenvalues $\lambda[\nabla^{2}u(x)]=(\lambda_1(x),\cdots,\lambda_n(x))\in\bar\Gamma_k$.
	
	A $C^2$ regular hypersurface $\mathcal{M}\subset\mathbb R^{n+1}$ is called $k$-convex if its principal curvature vector $\kappa(X)\in\Gamma_k^+$ for all $X\in\mathcal{M}$.
\end{definition}
For $k$-Hessian equation \eqref{eqn:1.1}, Xiao obtained asymptotic behavior of $k$-admissible solutio $u$ as follows.
\begin{lemma}[\cite{xiao}]\label{lem:3.1}
	Let $\Omega$ be a $(k-1)$-convex, star-shaped domain. Then, there exists a  $k$-admissible solution $u\in C^{1,1}(\mathbb R^n\setminus\Omega)$ to equation \eqref{eqn:1.1}, such that $\frac{u(x)}{-|x|^{2-\frac{n}{k}}}\rightarrow\rho$ as
	$|x|\rightarrow\infty$ in $C^2$ topology.
	Here, $\rho>0$ is a constant depending on $\gamma$. Moreover, for any $s \in [-1,0)$, the level set $\{u = s\}$ is regular.
\end{lemma}

\subsection{Monotonicity of $F(t)$}

\begin{proposition}\label{prop:3.1}
	Let $\Omega\subset\mathbb R^n$ be a $k$-convex and star-shaped domain with smooth boundary $\Sigma$.  Let
	\begin{equation*}
		\begin{aligned}
			C_1(t)=&(-t)^{-\frac{(a-k)(n-k)+k}{n-2k}}C_3+(-t)^{1-\frac{(a-k)(n-k)+k}{n-2k}}C_4,\\  C_2(t)=&-\frac{(a-k)(n-k)+k}{(n-2k)(a+1-k)}C_3(-t)^{-\frac{(a-k+1)(n-k)}{n-2k}}\\
			&-\frac{n-k}{n-2k}C_4(-t)^{1-\frac{(a-k+1)(n-k)}{n-2k}},
		\end{aligned}
	\end{equation*}
	where $C_3,C_4\in\mathbb R$ and $C_1(t)\geq0$.  Then, for any $a\geq\frac{k(n-k-1)}{n-k}$
	\begin{equation}
		F(t)= C_1(t)\int_{\Sigma_t} H_k |\nabla u|^a+C_2(t)\int_{\Sigma_t} H_{k-1} |\nabla u|^{a+1}
	\end{equation}
	is non-increasing. Moreover, $F(t)$ is constant if and only if $\Omega$ is a ball.
\end{proposition}

\begin{proof}
	{\bf Step 1: Approximating}.
	Since $u$ is only $C^{1,1}$, we need to establish the monotone property of $F(t)$ via approximation.
	So we adopt Ma-Zhang's idea to prove it (see section 6 in \cite{MZ} for details). We consider the following approximating equation
	\begin{equation}\label{ae}
		\left\{
		\begin{aligned}
			S_{k}(u^{\ve}_{ij})&=f^{\ve} \ \ \  {\rm in}\ \mathbb R^n\setminus\Omega,\\
			u^{\ve}&=-1\ \ \ {\rm on}\ \partial\Omega,\\
			u^{\ve}&\rightarrow0 \ {\rm as }\ \ x\rightarrow\infty ,
		\end{aligned}\right.
	\end{equation}
	where $f^{\ve}=c_{n,k}\ve^2(|x|^2+\ve^2)^{-\frac n2-1}$.
	
	On the level set $\Sigma_{t,\ve}=\{u^{\ve}=t\}$, we define $F_{\ve}(t)$ as follows:
	\begin{equation}\label{eqn:A1}
		F_{\ve}(t) = C_{1}(t)\int_{\Sigma_{t, \ve}} H_k |\nabla u^{\ve}|^{a} + C_2(t)\int_{\Sigma_{t, \ve}} H_{k-1} |\nabla u^{\ve}|^{a+1}.
	\end{equation}
	Here $C_1(t)$ and $C_2(t)$ satisfied the system of ODEs
	\begin{equation*}
		\begin{cases}
			C_2'(t)+(a-\frac{k(n-k-1)}{n-k})(\frac{n-k}{(n-2k)t})^2C_1(t)= 0, \\
			C_1'(t)-(a+1-k)C_2(t)+2\frac{n-k}{(n-2k)t}(a-\frac{k(n-k-1)}{n-k})C_1(t)= 0.
		\end{cases}
	\end{equation*}
	
	It is easy to see that
	\begin{equation}\label{eqn:A2}
		H_k=\frac{f^{\ve}}{|\nabla u^{\ve}|^k}-\frac{S_k^{ij}(\nabla^{2}u^{\ve})u_i^{\ve}u_l^{\ve}u^{\ve}_{lj}}{|\nabla u^{\ve}|^{k+2}},
	\end{equation}
	\begin{equation}\label{eqn:A3}
		H_{k-1}=\frac{S_{k}^{ij}(\nabla^{2}u^{\ve})u^{\ve}_iu^{\ve}_j}{|\nabla u^{\ve}|^{k+1}}.
	\end{equation}
	Using \eqref{eqn:A2} and \eqref{eqn:A3}, we can rewrite
	\begin{equation}\label{eqn:A4}
		\begin{aligned}
			F_{\ve}(t) =&\int_{\Sigma_{t, \ve}}C_{1}(u^{\ve}) \left(\frac{f^{\ve}}{|\nabla u^{\ve}|^k}-\frac{S_k^{ij}(\nabla^{2}u^{\ve})u_i^{\ve}u_l^{\ve}u_{lj}^{\ve}}{|\nabla u^{\ve}|^{k+2}}\right)|\nabla u^{\ve}|^{a}\\
			& +\int_{\Sigma_{t, \ve}} C_2(u^{\ve})S_{k}^{ij}(\nabla ^2u^{\ve})u_i^{\ve}u_j^{\ve} |\nabla u^{\ve}|^{a-k}.
		\end{aligned}
	\end{equation}
	
	Let $X_{\ve}=U_{\ve}+V_{\ve}$, where
	\begin{equation*}
		\left\{
		\begin{aligned}
			U_{\ve}=&C_{1}(u^{\ve}) f^{\ve}|\nabla u^{\ve}|^{(a-k-1)}\nabla u^{\ve},\\
			V_{\ve}=&-C_1(u^{\ve})S_k^{ij}(\nabla^{2}u^{\ve})u_l^{\ve}u_{lj}^{\ve}|\nabla u^{\ve}|^{a-k-1}\partial x^{i} \\
			&+C_2(u^{\ve})S_{k}^{ij}(\nabla ^2u^{\ve})u_j^{\ve} |\nabla u^{\ve}|^{a-k+1}\partial x^{i}.
		\end{aligned}\right.
	\end{equation*}
	
	Then
	\begin{equation}\label{eqn:A4}
		F_{\ve}(t)=\int_{\Sigma_{t, \ve}}\left\langle X_{\ve},\, \frac{\n u^{\ve}}{|\n u^{\ve}|}\right\rangle.
	\end{equation}
	Let $t_1 < t_2 $ be two regular value of $u^{\ve}$. One sees from the divergence theorem and \eqref{eqn:A4} that
	\begin{equation}\label{eqn:A5}
		\begin{aligned}
			F_{\ve}(t_2)-F_{\ve}(t_1)=&\int_{\{t_1<u^{\ve}<t_2\}}{\rm div}\ X_{\ve}.
		\end{aligned}
	\end{equation}
	{\bf Step 2: The estimation of \eqref{eqn:A5}}.		
	For simplicity, we use $u$ instead of $u^{\ve}$ and $S_k^{ij}$ instead of $S_k^{ij}(\nabla^{2}u^{\ve})$ during the
	proof.

	It is directing computation that
	\begin{equation*}
		\begin{aligned}
			{\rm div}U_{\ve}=&C_1'(u)f^{\ve}|\nabla u|^{(a-k+1)}+(f^{\ve})'C_{1}(u)|\nabla u|^{(a-k-1)}\langle\nabla u,x\rangle\\
			&+C_{1}(u) f^{\ve}{\rm div}(|\nabla u|^{(a-k-1)}\nabla u),
		\end{aligned}
	\end{equation*}
	\begin{equation*}
		\begin{aligned}
			{\rm div}V_{\ve}=&-\left(C_1'(u)S_k^{ij}u_lu_{lj}u_i+C_1(u)S_k^{ij}u_{il}u_{lj}+C_1(u)S_k^{ij}u_{l}u_{ijl}\right)|\nabla u|^{a-k-1}\\
			&-\left(a-\frac{k(n-k-1)}{n-k}-\frac n{n-k}\right)C_1(u)S_k^{ij}u_{l}u_{lj}u_{si}u_s|\nabla u|^{a-k-3}\\
			&+\left(C_2'(u)S_k^{ij}u_iu_j+C_2(u)S_k^{ij}u_{ij}\right)|\nabla u|^{a-k+1}\\
			&+(a-k+1)C_2(u)S_k^{ij}u_ju_{si}u_s|\nabla u|^{a-k-1}.
		\end{aligned}
	\end{equation*}
	Then, we have
	\begin{equation*}
		\begin{aligned}
			{\rm div}X_{\ve}=&-C_1(u)\left(S_k^{ij}u_{il}u_{lj}-\frac{n}{n-k}{S_{k}^{ij}}\partial_i|\nabla u|\partial_j|\nabla u|\right)|\nabla u|^{a-k-1}\\
			&-(a-\frac{k(n-k-1)}{n-k})C_1(u)S_k^{ij}\left(\partial_i|\nabla u|-\frac{n-k}{n-2k}\frac{|\nabla u|}{u}u_i\right)\\
			&\times\left(\partial_j|\nabla u|-\frac{n-k}{n-2k}\frac{|\nabla u|}{u}u_j\right)|\nabla u|^{a-k-1}\\
			&+\left(C_2'(u)+(a-\frac{k(n-k-1)}{n-k})(\frac{(n-k)}{(n-2k)u})^2C_1(u)\right)S_k^{ij}u_iu_j |\nabla u|^{a-k+1}\\
			&-\left(C_1'(u)-(a+1-k)C_2(u)+2\frac{n-k}{(n-2k)u}(a-\frac{k(n-k-1)}{n-k})C_1(u)\right)\\
			&\times S_k^{ij}\partial_i|\n u|u_j |\nabla u|^{a-k} \\
			& +{\rm div}U_{\ve}-C_1(u)S_k^{ij}u_{l}u_{ijl}|\nabla u|^{a-k-1}+C_2(u)S_k^{ij}u_{ij}|\nabla u|^{a-k+1} \\
			&\leq {\rm div}U_{\ve}-C_1(u)S_k^{ij}u_{l}u_{ijl}|\nabla u|^{a-k-1}+C_2(u)S_k^{ij}u_{ij}|\nabla u|^{a-k+1}\\
&-C_1(u)\left(S_k^{ij}u_{il}u_{lj}-\frac{n}{n-k}{S_{k}^{ij}}\partial_i|\nabla u|\partial_j|\nabla u|\right)|\nabla u|^{a-k-1}.
		\end{aligned}
	\end{equation*}
	Here, we have used $C_1(u)$ and $C_2(u)$ that satisfied ODEs System in Appendix A.

By using the uniform $C^0$-$C^2$ estimates of $u^{\ve}$ (see \cite{MZ}, \cite{xiao}), we have
	\begin{equation*}
		\begin{aligned}
			& \int_{t_1<u^{\ve}<t_2}{\rm div}U_{\ve}-C_1(u)S_k^{ij}u_{l}u_{ijl}|\nabla u|^{a-k-1}+C_2(u)S_k^{ij}u_{ij}|\nabla u|^{a-k+1} \\
			\leq& O(\ve^2)(|t_2|^{\frac {2k}{n-2k}}-|t_1|^{\frac{2k}{n-2k}}),
		\end{aligned}
	\end{equation*}
	where $t_1,t_2\in[-1,0)$. Here, we have used $$
	S_{k}^{ij}u_{ijl}=O(\ve^2)(|x|^2+\ve^2)^{-\frac n2-2}x_l\ \ {\rm for }\ l=1,\ldots,n.
	$$
	Thus, combining with Lemma \ref{lem:A1} in  Appendix A ,  we obtain that
	\begin{equation}\label{eqn:3.9}
		\int_{\{t_1<u^{\ve}<t_2\}}{\rm div}X_{\ve}\leq O(\ve^2)(|t_2|^{\frac {2k}{n-2k}}-|t_1|^{\frac{2k}{n-2k}}),
	\end{equation}
	for $t_1<t_2\in[-1,0)$.
	
	Let $\ve\rightarrow0$, from \eqref{eqn:A5} and Lemma \eqref{eqn:3.9}, it is easy to obtain the assertion. Moreover, if $F(t)$
	is constant,  we can see that $\Sigma_t$ is umbilical hypersurfaces.
\end{proof}


\subsection{Asymptotic behavior}
\begin{lemma}\label{asymptoticb}
Let $u$ be a  $k$-admissible $C^{1,1}$ solution of equation \eqref{eqn:1.1}. Along the level set $\Sigma_t=\{u=t\}$, we have that $t\rightarrow 0$
\begin{equation*}
	\begin{aligned}
		|\Sigma_{t}| =& |\mathbb S^{n-1}|(-t\rho^{-1})^{\frac{k(n-1)}{2k-n}} \left( 1 + o(1) \right),
		\\\int_{\Sigma_t} H_k|\nabla u|^{a} =&\binom{n-1}{k-1}(\frac{n}{k}-1)[(\frac{n}{k}-2)\rho]^{a}|\rho|^{\frac{(a-k)(k-n)-k}{n-2k}}|\mathbb S^{n-1}|\times\\
		&\times (-t)^{\frac{(a-k)(k-n)-k}{2k-n}}(1+o(1)),\\
		\int_{\Sigma_t}H_{k-1}|\nabla u|^{a+1}= & \binom{n-1}{k-1}[(\frac{n}{k}-2)\rho]^{a+1}|\rho|^{-\frac{(a+1-k)(n-k)}{n-2k}}|\mathbb S^{n-1}|\times\\
		&\times (-t)^{\frac{(a+1-k)(n-k)}{n-2k}}(1+o(1)).
	\end{aligned}
\end{equation*}
\end{lemma}
\begin{proof}
From Lemma \ref{lem:3.1} we know that
\begin{equation}\label{asymptotic1}
	u=-\rho |x|^{2-\frac{n}{k}}+o(|x|^{2-\frac{n}{k}}),
\end{equation}
\begin{equation}\label{asymptotic2}
	u_i=-(2-\frac{n}{k})\rho |x|^{-\frac{n}{k}}x_i+o(|x|^{1-\frac{n}{k}}),
\end{equation}
\begin{equation}\label{asymptotic3}
	u_{ij}=(2-\frac{n}{k})\rho |x|^{-\frac{n}{k}-2}(\frac{n}{k}x_ix_j-|x|^2\delta_{ij})+o(|x|^{-\frac{n}{k}}).
\end{equation}
By using \eqref{asymptotic1}-\eqref{asymptotic3}, we can obtain from \eqref{eqn:2.4} and \eqref{eqn:2.5} that
\begin{equation*}\label{asymptotic4}
	\left\{
	\begin{aligned}
		H_k=&-\frac{S_{k}^{ij}u_{im}u_mu_j}{|\nabla u|^{k+2}}=\binom{n-1}{k-1}(\frac{n}{k}-1)|x|^{-k}+o(|x|^{-k}),\\
		H_k|\nabla u|^{a}
		=&\binom{n-1}{k-1}(\frac{n}{k}-1)[(\frac{n}{k}-2)\rho]^{a}|x|^{a(1-\frac{n}{k})-k}+o(|x|^{a(1-\frac{n}{k})-k}),\\
		H_{k-1}|\nabla u|^{a+1}=&{S_{k}^{ij}u_iu_j}{|\nabla u|^{a-k}}\\
		=&\binom{n-1}{k-1}[(\frac{n}{k}-2)\rho]^{a+1}|x|^{(a+1)(1-\frac{n}{k})+1-k}+o(|x|^{(a+1)(1-\frac{n}{k})+1-k}).
	\end{aligned}\right.
\end{equation*}

Hence, as $t$ large enough, $H>0$ along $\Sigma_t$. It follows that $\Sigma_{t}$ is area outer-minimizing in $\mathbb R^n\setminus\Omega$ when $t$ is large enough, since the exterior of $\Sigma_{t}$ in $\mathbb R^n\setminus\Omega$ is foliated by mean-convex surfaces. Similarly, each coordinate sphere $S_r$ is area outer-minimizing in $\mathbb R^n\setminus\Omega$ when $r$ is large enough.
Denote $r_{-}(t) = \min\limits_{\Sigma_t} |x|$ and $r_{+}(t) = \max\limits_{\Sigma_t} |x|$. It follows from the outer-minimizing property that \begin{eqnarray}\label{xeq-limit4}
	|S_{r_{-}(t)}| \le |\Sigma_t| \le |S_{r_{+}(t)}|.
\end{eqnarray}
Note that
\begin{eqnarray}\label{xeq-limit5}
	|S_r|=|\mathbb S^{n-1}| r^{n-1}+o(r^{n-1}).
\end{eqnarray}
As $t\rightarrow0$, It follows from \eqref{asymptotic1} that
\begin{equation}\label{asymptotic5}
	|x|=(-t)^{\frac{k}{2k-n}}\rho^{\frac{k}{n-2k}}(1+o(1)).
\end{equation}
Therefore Lemma \ref{asymptoticb} can be obtained by above facts.
\end{proof}

\vskip 2mm
\subsection{Proof of Theorem \ref{thm:1.2} and \ref{thm:1.3}}~

{\bf The proof of Theorem \ref{thm:1.2}:}
From Lemma \ref{asymptoticb}, we have
\begin{equation}
\begin{aligned}
	\lim\limits_{t\rightarrow0}F(t)=&\left(\frac{n}{k}-\frac{(a-k)(n-k)+k}{k(a+1-k)}\right)\binom{n-1}{k-1}(\frac{n}{k}-2)^{a}
	\rho^{ k(n-k-a-1)}|\mathbb S^{n-1}|C_3\\
	=&\frac{n-2k}{k(a+1-k)}\binom{n-1}{k-1}(\frac{n}{k}-2)^{a}
	\rho^{ k(n-k-a-1)}|\mathbb S^{n-1}|C_3.
\end{aligned}
\end{equation}
This combining with Proposition \ref{prop:3.1} yields to
\begin{equation*}
F(t)\geq\lim\limits_{t\rightarrow0}F(t)=\frac{n-2k}{k(a+1-k)}\binom{n-1}{k-1}(\frac{n}{k}-2)^{a}
\rho^{ k(n-k-a-1)}|\mathbb S^{n-1}|C_3,
\end{equation*}
for any $t\in[-1,0)$. The equality holds if and only if $\Omega$ is a ball.

{\bf The proof of Theorem \ref{thm:1.3}:} Let $C_3=0,C_4=1$ in  \eqref{eqn:1.8}, we can obtain \eqref{eqn:1.9}.
Let $C_3=1,C_4=-1, a=n-k-1$ in  \eqref{eqn:1.8}, we can obtain \eqref{eqn:1.10}.  The each equality holds if and only if $\Omega$ is a ball.

\section{The proof of Theorem \ref{thm:1.1}}\label{sec4}

In order to prove Theorem \ref{thm:1.1}, we need to prove some integral identities.
%

Firstly, we follow Ma-Zhang and Xiao's idea (for detail see \cite{MZ} or \cite{xiao}) to prove the following Lemmas.
\begin{lemma}\label{lem:3.3}
	Let $\Omega\subset\mathbb R^n$ be a bounded, $k-$convex star-shaped domain with boundary of class $ C^{2}$ with $2\leq k<\frac n2$.
	Then, the solution $u$ to \eqref{eqn:1.1} and \eqref{eqn:1.2} satisfies
	\begin{equation}\label{eqn:3.05}
		\begin{aligned}
			(k+1)\int_{\mathbb R^n\setminus\bar\Omega}S_{k-1}|\nabla u|^{2}dx+\int_{\partial\Omega}H_{k-2}|\nabla u|^{k+1} d\sigma-2c^k\int_{\partial \Omega}H_{k-1}d\sigma=0.
		\end{aligned}
	\end{equation}
	
\end{lemma}
\begin{proof}
	We consider the approximate equation \eqref{ae}, then we have
	\begin{equation}\label{overdeterminedb}
		\lim\limits_{\ve\rightarrow 0}|\nabla u^{\ve}|=|\nabla u|=c\ \ \ {\rm on}\ \partial\Omega.
	\end{equation}
	For simplicity, we use $u$ instead of $u^{\ve}$ and $S_k(u_{ij})$ instead of $S_k(u^{\ve}_{ij})$ during the
	proof. We adopt Ma-Zhang's idea to prove it (see section 6 in \cite{MZ} for details).
	
	On the $\partial\Omega$, we have
	\begin{equation*}
		H_{k-2}|\nabla u|^{k+1}=S_{k-1}^{ij}u_i\frac{u_j}{|\n u|}|\nabla u|^2=|\nabla u|^2S_{k-1}^{ij}u_i\nu_j.
	\end{equation*}
	Applying the divergence theorem, we get
	\begin{equation*}
		\begin{aligned}
			-\int_{\partial\Omega}H_{k-2}|\nabla u|^{k+1}d\sigma=&\int_{B_R\setminus\bar\Omega}(S_{k-1}^{ij}u_i|\nabla u|^2)_jdx-\int_{\partial B_R}S_{k-1}^{ij}u_i\nu_{\partial B_R}^j|\nabla u|^2d\sigma\\
			=&-\int_{\partial B_R}S_{k-1}^{ij}u_i\nu_{\partial B_R}^j|\nabla u|^2d\sigma+(k-1)\int_{B_R\setminus\bar\Omega}S_{k-1}|\nabla u|^2dx\\
			&+2\int_{B_R\setminus\bar\Omega}S_{k-1}^{ij}u_iu_{mj}u_mdx.
		\end{aligned}
	\end{equation*}
	Since
	\begin{equation*}
		S_k^{im}u_iu_m=S_{k-1}|\nabla u|^2-S_{k-1}^{ij}u_{mj}u_iu_m,
	\end{equation*}
	we have
	\begin{equation*}
		\begin{aligned}
			-\int_{\partial\Omega}H_{k-2}|\nabla u|^{k+1}d\sigma
			=&-\int_{\partial B_R}S_{k-1}^{ij}u_i\nu_{\partial B_R}^j|\nabla u|^2d\sigma+(k+1)\int_{B_R\setminus\bar\Omega}S_{k-1}|\nabla u|^2dx\\
			&-2\int_{B_R\setminus\bar\Omega}S_{k}^{ij}u_iu_jdx\\
			=&\int_{\partial B_R}S_{k-1}^{ij}u_i\nu_{\partial B_R}^j|\nabla u|^2d\sigma+(k+1)\int_{B_R\setminus\bar\Omega}S_{k-1}|\nabla u|^2dx\\
			&-2\int_{\partial B_R}uS_{k}^{ij}u_i\nu_{\partial B_R}^jd\sigma+2\int_{\partial \Omega}uS_{k}^{ij}u_i\nu^jd\sigma.
		\end{aligned}
	\end{equation*}
	By the $C^0$- $C^2$ estimate of $u^{\ve}$ in $\mathbb R^n\setminus\Omega$ (for details see \cite{MZ}, \cite{xiao}), we have
	\begin{equation*}
		\begin{aligned}
			&\lim_{R\rightarrow\infty}\int_{\partial B_R}S_{k-1}^{ij}u_i\nu_{\partial B_R}^j|\nabla u|^2d\sigma=0,\\
			&\lim_{R\rightarrow\infty}\int_{\partial B_R}uS_{k}^{ij}u_i\nu_{\partial B_R}^jd\sigma=0,
		\end{aligned}
	\end{equation*}
	and $\lim_{\ve\rightarrow0}|\nabla u^{\ve}|=|\nabla u|$ on $\partial\Omega$.
	
	Finally, by taking the limit for  $\ve\rightarrow0$ and $R\rightarrow+\infty$, we obtain the assertion.
\end{proof}

\subsection{Rellich-Pohozaev-type identity}\

We will prove the Rellich-Pohozaev-type identity by using the same method with Lemma \ref{lem:3.3}.
\begin{lemma}\label{lem:3.4}
	Let $u$ be a $C^{1,1}$ solution of problem \eqref{eqn:1.1} and \eqref{eqn:1.2}. Then for $k\ge2$ we have that
	\begin{equation}\label{eqn:3.6}
		\begin{aligned}
			0=&(n-k+1)\left(\int_{\mathbb R^n\setminus\bar\Omega}S_{k-1}|\nabla u|^2dx+\frac{c^{k+1}}{k-1}\int_{\partial\Omega}H_{k-2} d\sigma\right)\\
			&- \frac{2(n-k)c^k}{k}\int_{\partial\Omega}H_{k-1} d\sigma.
		\end{aligned}
	\end{equation}
	
\end{lemma}
\begin{proof}
	Since $S_{k}(\nabla^2 u)=f^{\ve}$, we have
	$$
	S_{k}^{ij}u_{ijl}=O(\ve^2)(|x|^2+\ve^2)^{-\frac n2-2}x_l\ \ {\rm for }\ l=1,\ldots,n.
	$$
	Moreover, by using Proposition \ref{prop:2.1} (ii), we have
	\begin{equation}\label{eqn:3.06}
		\begin{aligned}
			kf^{\ve}u=&kS_k(\nabla^2 u)u=S_{k}^{ij}u_{ij}u=\frac 12S_{k}^{ij}u_{il}u(|x|^2)_{lj}\\
			=&\frac12\left(S_{k}^{ij}u_{il}u(|x|^2)_{l}\right)_j-\frac12S_{k}^{ij}u_{ilj}u(|x|^2)_{l}-\frac12S_{k}^{ij}u_{il}u_j(|x|^2)_{l}\\
			=&\frac12\left(S_{k}^{ij}u_{il}u(|x|^2)_{l}\right)_j-\frac12S_{k}^{il}u_{ij}u_j(|x|^2)_{l}+O(\ve^2)u(|x|^2+\ve^2)^{-\frac n2-1}\\
			=&\frac12\left(S_{k}^{ij}u_{il}u(|x|^2)_{l}\right)_j-\frac12\left(S_{k}^{il}|\nabla u|^2x_l\right)_i+\frac{(n-k+1)}2S_{k-1}|\nabla u|^2\\
			&+O(\ve^2)u(|x|^2+\ve^2)^{-\frac n2-1}.
		\end{aligned}
	\end{equation}
	By directing computations, we get
	\begin{equation}\label{eqn:3.99}
		\begin{aligned}
			O(\ve^2)(r_0^{-\frac nk}-R^{-\frac nk})=&\int_{B_R\setminus\Omega}\left(S_{k}^{ij}u_{il}u(|x|^2)_{l}\right)_j-\left(S_{k}^{il}|\nabla u|^2x_l\right)_idx\\
			&+\int_{B_R\setminus\Omega}(n-k+1)S_{k-1}|\nabla u|^2dx\\
			=&\int_{\partial B_R}S_{k}^{ij}u_{il}u(|x|^2)_{l}\nu^j_{\partial B_R}d\sigma-\int_{\partial \Omega}S_{k}^{ij}u_{il}u(|x|^2)_{l}\nu^jd\sigma\\
			&-\int_{\partial B_R}S_{k}^{il}x_l\nu^i_{\partial B_R}|\nabla u|^2d\sigma+\int_{\partial \Omega}S_{k}^{il}x_l\nu^i|\nabla u|^2d\sigma\\
			&+\int_{B_R\setminus\Omega}(n-k+1)S_{k-1}|\nabla u|^2d\sigma.
		\end{aligned}
	\end{equation}
	Here, we have used $C^0$ estimate of $u^{\ve}$  such that
	$$
	\int_{B_R\setminus\Omega}u(|x|^2+\ve^2)^{-\frac n2-1}=O(\ve^2)(r_0^{-\frac nk}-R^{-\frac nk}).
	$$

	Then, by using Proposition \ref{prop:2.1} (ii), we have
	\begin{equation}\label{eqn:3.07}
		\begin{aligned}
			\int_{\partial \Omega}S_{k}^{ij}u_{il}u(|x|^2)_{l}\nu^j d\sigma
			&=-2\int_{\partial \Omega}S_{k}^{ij}u_{il}x_{l}\nu^jd\sigma\\
			&=2\int_{\partial \Omega}S_{k+1}^{jl}x_{l}\nu^jd\sigma-2\int_{\partial\Omega}f^{\ve}\langle x,\,\nu\rangle.
		\end{aligned}
	\end{equation}
	At any point $p$ satisfying $\nabla u\not=0$, we choose $\{e_{\alpha}\}_{\alpha=1}^{n-1} $such that
	$$
	u_{\alpha\beta}=\lambda_{\alpha}\delta_{\alpha\beta} \ \ {\rm and }\ e_n=\frac{\nabla u}{|\nabla u|}.
	$$
	On $\partial\Omega$, we have
	\begin{equation}\label{eqn:3.77}
		\begin{aligned}
			S_{k+1}^{ij}x_i\nu_j=&S_{k+1}^{in}x_i=S_{k+1}^{nn}x\cdot\nu +S_{k+1}^{\alpha n}x\cdot e_{\alpha}\\
			=&S_{k+1}^{ij}u_iu_jx\cdot\nu|\nabla u|^{-2}-S_{k-1}(\lambda|\alpha)x\cdot e_{\alpha}u_{n\alpha}\\
			=&H_{k}|\nabla u|^{k}x\cdot\nu-S_{k-1}(\lambda|\alpha)x\cdot e_{\alpha}u_{n\alpha}+O(\ve^2).
		\end{aligned}
	\end{equation}
	Similarly,
	\begin{equation}\label{eqn:3.11}
		\begin{aligned}
			S_{k}^{ij}x_i\nu_j|\nabla u|^2
			=&S_{k}^{ij}u_iu_jx\cdot\nu-S_{k-2}(\lambda|\alpha)x\cdot e_{\alpha}u_{n\alpha}|\nabla u|^2\\
			=&H_{k-1}|\nabla u|^{k+1}x\cdot\nu-S_{k-2}(\lambda|\alpha)x\cdot e_{\alpha}u_{n\alpha}|\nabla u|^2.
		\end{aligned}
	\end{equation}
	
	By using the $C^0$-$C^2$ estimate of $u^{\ve}$, we know that $S_{k-2}(\lambda|\alpha)x\cdot e_{\alpha}u_{n\alpha}|\nabla u|^2$ and $S_{k-1}(\lambda|\alpha)x\cdot e_{\alpha}u_{n\alpha}|\nabla u|^2$ is  bounded. By  Lebesgue dominated convergence theorem, we have
	\begin{equation*}
		\begin{aligned}
			&\lim_{\ve\rightarrow0}\int_{\partial\Omega}S_{k-1}(\lambda|\alpha)x\cdot e_{\alpha}u_{n\alpha}|\nabla u|^2=0,
			&\lim_{\ve\rightarrow0}\int_{\partial\Omega}S_{k-2}(\lambda|\alpha)x\cdot e_{\alpha}u_{n\alpha}|\nabla u|^2=0.
		\end{aligned}
	\end{equation*}
	Here, we used $\lim\limits_{\ve\rightarrow0}u_{n\alpha}^2=0$ on $\partial\Omega$.
	
	We can also see that
	\begin{equation*}
		\begin{aligned}
			&\lim_{R\rightarrow\infty}\int_{\partial B_R}S_{k}^{ij}u_{il}u(|x|^2)_{l}\nu^j_{\partial B_R}d\sigma=0,\ \ \
			\lim_{R\rightarrow\infty}\int_{\partial B_R}S_{k}^{il}x_l\nu^i_{\partial B_R}|\nabla u|^2d\sigma=0.
		\end{aligned}
	\end{equation*}
	
	Substituting \eqref{eqn:3.77}-\eqref{eqn:3.11} into \eqref{eqn:3.99}, then taking the limits for $
	\ve\rightarrow0$  and  $R\rightarrow+\infty$, we obtain that
	$$
	0=c^{k+1}\int_{\partial \Omega}H_{k-1}x\cdot\nu d\sigma
	+\int_{\mathbb R^n\setminus\bar\Omega}(n-k+1)S_{k-1}|\nabla u|^2d\sigma-2c^k\int_{\partial \Omega}H_{k}x\cdot\nu d\sigma.\\
	$$
	This combining with  Minkowskian integral formulae yields to the assertion.
	%
\end{proof}

{\bf The proof of Theorem \ref{thm:1.1}:}
It follows from Lemma \ref{lem:3.3} and \ref{lem:3.4} that when $k\ge 2$,
\begin{equation}\label{eqn:4.1}
	c=\frac{n-2k}k\frac{k-1}{n-k+1}\frac{\int_{\partial\O}H_{k-1}}{\int_{\partial\O}H_{k-2}}.
\end{equation}

On the other hand, by \eqref{eqn:1.9} of Theorem \ref{thm:1.3} and overdetermined condition, we obtain
\begin{equation}\label{eqn:4.2}
	\frac{n-k}{n-2k}c\leq\frac{\int_{\partial\O}H_{k}}{\int_{\partial\O}H_{k-1}}.
\end{equation}

From \eqref{eqn:4.1} and \eqref{eqn:4.2}, we obtain
\begin{equation*}
	(n-k)(k-1)\left(\int_{\partial\O}H_{k-1}\right)^2\leq(n-k+1)k\int_{\partial\O}H_{k}\int_{\partial\O}H_{k-2}.
\end{equation*}

This combining with special Aleksandrov-Fenchel inequalities (see Notes for Section 7.4 of \cite{S})

\begin{equation*}
	(n-k)(k-1)\left(\int_{\partial\O}H_{k-1}\right)^2\geq(n-k+1)k\int_{\partial\O}H_{k}\int_{\partial\O}H_{k-2}
\end{equation*}
yields to
\begin{equation*}
	(n-k)(k-1)\left(\int_{\partial\O}H_{k-1}\right)^2=(n-k+1)k\int_{\partial\O}H_{k}\int_{\partial\O}H_{k-2},
\end{equation*}
which implies that $\Omega$ is ball.

If $k=1$, it follows from Brandolini-Nitsch-Salani's result \cite{BNS} that $\Omega$ is a ball. We should point out that they have used the Minkowski-type inequality
\begin{equation}\label{eqn:4.11} 
	|\Omega|\int_{\partial\Omega}H\leq\frac{n-1}{n}|\partial\Omega|^2
\end{equation}
by using method of convex geometry.  Qiu-Xia  (\cite{QX}) give an  alternative approach  to prove this Minkowski-type inequality.

This complete the proof of Theorem \ref{thm:1.1}.
\begin{remark}
	When $k=1$, inequality \eqref{eqn:4.11} ( or \eqref{eqn:4.2}) holds without convex condition (see \cite{XY}). 
\end{remark}

\begin{appendices}

\section{Solution of the ODEs system}\label{secA1}
In this appendix, we solve the ordinary differential equations for $t\in[-1, 0)$:
\begin{equation*}
	\begin{cases}
		C_2'(t)+(a-\frac{k(n-k-1)}{n-k})(\frac{n-k}{(n-2k)t})^2C_1(t)= 0, \\
		C_1'(t)-(a+1-k)C_2(t)+2\frac{n-k}{(n-2k)t}(a-\frac{k(n-k-1)}{n-k})C_1(t)= 0.
	\end{cases}
\end{equation*}
We can solve the solution of this ODEs system as folllows
\begin{equation*}
	\begin{aligned}
		C_1(t)=&(-t)^{-\frac{(a-k)(n-k)+k}{n-2k}}C_3+(-t)^{1-\frac{(a-k)(n-k)+k}{n-2k}}C_4,\\  C_2(t)=&-\frac{(a-k)(n-k)+k}{(n-2k)(a+1-k)}C_3(-t)^{-\frac{(a-k+1)(n-k)}{n-2k}}\\
		&-\frac{n-k}{n-2k}C_4(-t)^{1-\frac{(a-k+1)(n-k)}{n-2k}},
	\end{aligned}
\end{equation*}
where $C_3,C_4\in \mathbb R$.



\begin{lemma}\label{lem:A1}
Let $u^{\ve}$ be a $k$-admissible solution of \eqref{ae}. Then  we have, for $C_1(u^{\ve})\geq0$,
\begin{equation*}
\begin{aligned}
\int_{t_1<u^{\ve}<t_2}\left(S_k^{ij}u^{\ve}_{im}u^{\ve}_{mj}-\frac{n}{n-k}S_k^{ij}D_i|Du^{\ve}|D_j|Du^{\ve}|\right)
C_1(u^{\ve})|Du^{\ve}|^{a-k-1}\\
\geq- O(\ve^2)\left(|t_2|^{\frac{2k}{n-2k}}-|t_1|^{\frac{2k}{n-2k}}\right).
\end{aligned}
\end{equation*}
\end{lemma}

\begin{proof}
For simplicity, we use $u$ instead of $u^{\ve}$ and $S_k^{ij}$ instead of $S_k^{ij}(\nabla^{2}u^{\ve})$ during the
	proof.

Let $1\leq\alpha,\beta,\ldots\leq n-1$ and $1\leq 1,j,m,\ldots\leq n$.
Choosing $\{e_\alpha\}$ such that $u_{\alpha\beta}=\lambda_{\alpha}\delta_{\alpha\beta}$ and $e_n=\frac{\nabla u}{|\n u|}$.
Then we have
\begin{equation}\label{eqn:AA1}
\begin{aligned}
  S_k^{\alpha\alpha}
  =&S_{k-2}(\lambda|\alpha)u_{nn}+S_{k-1}(\lambda|\alpha)-\sum_{\beta\not=\alpha}S_{k-3}(\lambda|\alpha,\beta)u_{\beta n}^2,
  \end{aligned}
\end{equation}
and
\begin{equation}\label{eqn:AA2}
  \begin{aligned}
  S_k^{\alpha n}
=&-S_{k-2}(\lambda|\alpha)u_{\alpha n}.
  \end{aligned}
\end{equation}
We also can see that $\alpha\not=\beta$
\begin{equation}\label{eqn:AA3}
  \begin{aligned}
  S_k^{\alpha \beta}
=&0.
  \end{aligned}
\end{equation}
Noticing that
\begin{equation}\label{eqn:AA4}
  kS_k(\lambda)=\sum_{\alpha}S_{k-1}(\lambda|\alpha)\lambda_{\alpha}.
\end{equation}

It follows from that $S_{k}(D^2u^{\ve})=f^{\ve}$ that
\begin{equation}\label{eqn:AA5}
\begin{aligned}
f^{\ve}
=&S_{k-1}(\lambda)u_{nn}+S_k(\lambda)-S_{k-2}(\lambda|\alpha)u_{\alpha n}^2
\end{aligned}
\end{equation}

Similarly, we have
\begin{equation}\label{eqn:AA6}
\begin{aligned}
S_{k+1}(D^2u^{\ve})
=&S_{k}(\lambda)u_{nn}+S_{k+1}(\lambda)-S_{k-1}(\lambda|\alpha)u_{\alpha n}^2
\end{aligned}
\end{equation}
Next, we have the following computation
\begin{equation}\label{eqn:AA7}
\begin{aligned}
S_k^{ij}u_{in}u_{jn}=&S_{k}^{\alpha\alpha}u_{\alpha n}^2+2S_{k}^{\alpha n}u_{\alpha n}u_{nn}+S_k^{nn}u_{nn}^2\\
=&\Big(\big(S_{k-2}(\lambda|\alpha)u_{nn}+S_{k-1}(\lambda|\alpha)-\sum_{\beta\not=\alpha}
S_{k-3}(\lambda|\alpha,\beta)u_{\beta n}^2\big)u_{\alpha n}^2\\
&-2S_{k-2}(\lambda|\alpha)u_{\alpha n}^2u_{nn}+S_{k-1}(\lambda)u_{nn}^2\Big)\\
=&S_{k-1}(\lambda)u_{nn}^2-S_{k-2}(\lambda|\alpha)u_{\alpha n}^2u_{nn}+(S_{k-1}(\lambda|\alpha)\\
&-\sum_{\beta\not=\alpha}S_{k-3}(\lambda|\alpha,\beta)u_{\beta n}^2\big)u_{\alpha n}^2\\
\leq &(f^{\ve}-S_k(\lambda))u_{nn}+S_{k-1}(\lambda|\alpha)u_{\alpha n}^2
\end{aligned}
\end{equation}

By using \eqref{eqn:AA5}-\eqref{eqn:AA7}, we have
\begin{equation}\label{eqn:AA8}
\begin{aligned}
&-S_{k}^{ij}u_{ik}u_{kj}+\frac{n}{n-k}S_{k}^{ij}u_{in}u_{jn}\\
=&(k+1)S_{k+1}+\frac{n}{n-k}S_{k}^{ij}u_{in}u_{jn}\\
\leq&(k+1)S_{k+1}(\lambda)+\left((k+1)-\frac{n}{n-k}\right)S_k(\lambda)u_{nn}-(k+1-\frac{n}{n-k})S_{k-1}(\lambda|\alpha)u_{\alpha n}^2\\
&+\frac{n}{n-k}f^{\ve}u_{nn}\\
=&(k+1)S_{k+1}(\lambda)-\left(\frac{k(n-k-1)}{n-k}\right)S_k(\lambda)\frac{S_k(\lambda)-S_{k-2}(\lambda|\alpha)u_{\alpha n}^2}{S_{k-1}(\lambda)}\\
&-\left(\frac{k(n-k-1)}{n-k}\right)S_{k-1}(\lambda|\alpha)u_{\alpha n}^2+\frac{n}{n-k}f^{\ve}u_{nn}+\frac{k(n-k-1)}{n-k}\frac{S_k(\lambda)}{S_{k-1}(\lambda)}f^{\ve}\\
=&(k+1)S_{k+1}(\lambda)-\left(\frac{k(n-k-1)}{n-k}\right)\frac{S_k^2(\lambda)}{S_{k-1}(\lambda)}
+\frac{k(n-k-1)}{n-k}\frac{S_k(\lambda)}{S_{k-1}(\lambda)}f^{\ve}\\
&+\left(\frac{k(n-k-1)}{n-k}\right)\left(S_k(\lambda)\frac{S_{k-2}(\lambda|\alpha)u_{\alpha n}^2}{S_{k-1}(\lambda)}-S_{k-1}(\lambda|\alpha)u_{\alpha n}^2\right)+\frac{n}{n-k}f^{\ve}u_{nn}\\
\leq&\frac{n}{n-k}f^{\ve}|D^2u|+\frac{k(n-k-1)}{n-k}H_{k-1}^{\frac1{k-1}}f^{\ve}|Du|.
\end{aligned}
\end{equation}
By using the uniform $C^0-C^2$ estimates of $u^{\ve}$ (see \cite{MZ}, \cite{xiao}), we have
\begin{equation*}
\begin{aligned}
\left(\frac{n}{n-k}f^{\ve}|D^2u|+\frac{k(n-k-1)}{n-k}H_1|Du|f^{\ve}\right)C_1(u)|Du|^{a-k-1}\leq O(\ve^2)|x|^{-n-\frac nk}.
\end{aligned}
\end{equation*}
Thus \begin{equation*}
\begin{aligned}
\int_{t_1<u^{\ve}<t_2}\left(S_k^{ij}u^{\ve}_{im}u^{\ve}_{mj}-\frac{n}{n-k}S_k^{ij}D_i|Du^{\ve}|D_j|Du^{\ve}|\right)
C_1(u^{\ve})|Du^{\ve}|^{a-k-1}\\
\geq- O(\ve^2)\left(|t_2|^{\frac{2k}{n-2k}}-|t_1|^{\frac{2k}{n-2k}}\right).
\end{aligned}
\end{equation*}
\end{proof}

\end{appendices}

\bmhead{Acknowledgements}

The authors would like to thank  Professor  Hui Ma, Zhizhang Wang, Chao Xia and Doctor Mingxuan Yang for their helpful conversations on this work.

\bmhead{Funding}

Jiabin Yin was supported by the NSF of China (Grant Number 12201138) and Mathematics Tianyuan fund project (Grant No. 12226350). Xingjian Zhou was supported by the China Scholarship Council (No. 202406310160).



\begin{thebibliography}{99}
	%
	%
	\bibitem{AFM2020}
	V. Agostiniani, M. Fogagnolo, L. Mazzieri. Sharp geometric inequalities for closed
	hypersurfaces in manifolds with nonnegative Ricci curvature. Invent. Math. 222
	(2020), no. 3, 1033-1101.
	
	\bibitem{AFM}
	V. Agostiniani, M. Fogagnolo,  L. Mazzieri. Minkowski inequalities via nonlinear
	potential theory. Arch. Ration. Mech. Anal., 244(1):51-85, 2022.
	\bibitem{AM2017}
	V. Agostiniani, L. Mazzieri.  On the geometry of the level sets of bounded static potentials. Comm. Math. Phys. 355 (2017), no. 1, 261-301.
	
	\bibitem{AM}
	V. Agostiniani, L. Mazzieri. Monotonicity formulas in potential theory. Calc. Var.
	Partial Differential Equations, 59(1):Paper No. 6, 32, 2020.
	
	\bibitem{AMMO}
	V. Agostiniani, C. Mantegazza, L. Mazzieri, F. Oronzio. Riemannian Penrose inequality via Nonlinear Potential Theory, arXiv:2205.11642.
	
	\bibitem{AMO2022}
	V. Agostiniani, L. Mazzieri, F. Oronzio. A geometric capacitary inequality for sub-static manifolds with harmonic potentials. Math. Eng. 4 (2022), no. 2, Paper No. 013, 40 pp.
	
	\bibitem{AMO2024}
	V. Agostiniani, L. Mazzieri, F. Oronzio.  A Green's function proof of the positive mass theorem. Comm. Math. Phys. 405 (2024), no. 2, Paper No. 54, 23 pp
	
	\bibitem{Al}
	A. D. Alexandrov. Uniqueness theorem for surfaces in the large, V. Vestnik, Leningrad
	Univ. 13, 19 (1958), 5-8, Amer. Math. Soc. Transl. 21, Ser. 2, 412-416.
	\bibitem{Al1}
	A. D. Alexandrov. A characteristic property of spheres, Ann. Math. Pura Appl. 58 (1962),
	303-315.
	
	\bibitem{BFM2024}
	L. Benatti, M. Fogagnolo, L. Mazzieri. Minkowski inequality on complete Riemannian manifolds with nonnegative Ricci curvature. Anal. PDE 17 (2024), no. 9, 3039-3077.
	
	\bibitem{BNS}
	B. Brandolini, C. Nitsch, P. Salani. An overdetermined problem for the anisotropic capacity. Calc. Var. Partial Equations 55(2016), no. 4,  Art. 84, 24 pp.
	
	\bibitem{BNST}
	B. Brandolini, C. Nitsch, P. Salani, C. Trombetti. Serrin-type overdetermined problems: an alternative proof. Arch. Ration. Mech. Anal. 190 (2008), no. 2, 267-280.
	
	\bibitem{CNS}
	L. Caffarelli, L. Nirenberg, J. Spruck. The Dirichlet problem for nonlinear
	second-order elliptic equations. III. Functions of the eigenvalues of the Hessian. Acta
	Math. 155 (1985), 261-301.
	%
	\bibitem{DGX1}
	F. Della Pietra, N. Gavitone, C. Xia. Symmetrization with respect to mixed volumes. Adv. Math. 388(2021).
	%
	
	
	\bibitem{FM}
	M. Fogagnolo, L. Mazzieri. Minimising hulls, $p$-capacisty and isoperimetric inequality on complete Riemannian
	manifolds. J. Funct. Anal. 283 (2022), no. 9, Paper No. 109638, 49 pp.
	\bibitem{FMP}
	M. Fogagnolo, L. Mazzieri, A. Pinamonti. Geometric aspects of $p$-capacitary potentials. Ann. Inst. H. Poincar\'e C Anal. Non Lin\'eaire, 36(4):1151-1179, 2019.

	%
	\bibitem{Garding}
	L. Garding. An inequality for hyperbolic polynomials. J. Math. Mech. 8 (1959): 957-65.
	\bibitem{GS}
	N. Garofalo, E. Sartori. Symmetry in exterior boundary value problems for quasilinear elliptic equations via
	blow-up and apriori estimates. Adv Differ Equ. 4 (1999), 137-161.
	%
	
	\bibitem{Guan-Li}
	P. Guan, J. Li. The quermassintegral inequalities for $k$-convex star-shaped domains.
	Adv. Math. 221, no. 5 (2009): 1725-32.
	
	\bibitem{H}
	C.C. Hsiung. Some integral formulas for closed hypersurfaces. Math. Scand. 2, (1954)
	286-294.

	%
	
	\bibitem{MZ}
	X.N. Ma, D.K. Zhang. The exterior dirichlet problem for the homogeneous $k$-Hessian equation. arXiv:2207.13504v1
	
	\bibitem{Miao2024}
	P.Z. Miao. Implications of some mass-capacity inequalities. J. Geom. Anal. 34 (2024), no. 8, Paper No. 241, 16 pp.
	
	\bibitem{P}
	G. Poggesi. Radial symmetry for $p$-harmonic functions in exterior and punctured domains.
	Appl. Anal. 98 (2019), no. 10, 1785-1798.
	%
	\bibitem{QX}
	G.H. Qiu, C. Xia. Classical Neumann problems for Hessian equations and Alexandrov-Fenchel's inequalities. Int. Math. Res. Not. IMRN 2019, no. 20, 6285-6303.
	\bibitem{R}
	W. Reichel.  Radial symmetry for elliptic boundary-value problems on exterior
	domains. Arch. Rational Mech. Anal. 137 (1997), no. 3, 381-394.
	
	\bibitem{R1}
	W. Reichel. Radial symmetry for an electrostatic, a capillarity and some fully nonlinear overdetermined problems
	on exterior domains. Z Anal Anwend. 15 (1996),619-635.
	
	\bibitem{Reilly} R.C. Reilly. On the Hessian of a function and the curvatures of its graph.  Michigan
	Math. J., 20 (1973) 373-383.
	
	%
	\bibitem{S}
	R. Schneider. Convex bodies: the Brunn-Minkowski theory. Encyclopedia of Mathematics and its Applications, vol. 44. Cambridge University Press, Cambridge, 1993
	
	
	\bibitem{Se}
	J. Serrin. A symmetry problem in potential theory. Arch. Ration. Mech. Anal. 43 (1971),
	304-318.
	
	\bibitem{T}
	N. S. Trudinger. On new isoperimetric inequalities and symmetrization. J. Reine
	Angew. Math. 488 (1997), 203-220.
	%
	%
	\bibitem{Wei}
	H.F. Weinberger. Remark on the preceding paper of Serrin. Arch. Ration. Mech. Anal. 43 (1971), 319-320.
	\bibitem{XY}
	C. Xia, J.B. Yin. Two overdetermined problems for anisotropic p-Laplacian. Math. Eng. 4 (2022), no. 2, Paper No. 015, 18 pp.
	\bibitem{XYZ}
	C. Xia, J.B. Yin, X.J. Zhou. New monotonicity for $p$-capacitary functions in $3$-manifolds with nonnegative scalar curvature. Adv. Math. 440 (2024), Paper No. 109526, 40 pp.
	\bibitem{xiao}
	L. Xiao.  Generalized Minkowski inequality via degenerate $k$-Hessian equation on the exterior domain. arXiv:2207.05673v1
	
	%
\end{thebibliography}
\end{document}